\documentclass[12pt, a4paper]{article}
\usepackage{amsmath, amsthm, amscd, amsfonts, amssymb, graphicx, color}
\usepackage[bookmarksnumbered,colorlinks, plainpages]{hyperref}

\newtheorem{theorem}{Theorem}[section]
\newtheorem{lemma}[theorem]{Lemma}
\newtheorem{proposition}[theorem]{Proposition}

\theoremstyle{definition}

\newtheorem{example}[theorem]{Example}

\theoremstyle{remark}

\numberwithin{equation}{section}
\makeatletter
\newcommand{\vast}{\bBigg@{3}}
\newcommand{\Vast}{\bBigg@{4}}
\makeatother

\begin{document}

\begin{center} \begin{Huge}\textbf{Strong Twins of Ordinary Star-Like Self-Contained Graphs}\end{Huge} \end{center}

\begin{center}\textbf{Mohammad Hadi Shekarriz\footnote{\small{Corresponding author\\ \indent \textit{E-mail addresses:} mh.shekarriz@stu.um.ac.ir and mirzavaziri@um.ac.ir}} \textbf{and Madjid Mirzavaziri}}

Department of Pure Mathematics, Ferdowsi University of Mashhad,\\
P. O. Box 1159, Mashhad 91775, Iran. \end{center}

\begin{abstract}
A self-contained graph is an infinite graph which is isomorphic to one of its proper induced subgraphs. In this paper, ordinary star-like self-contained graphs are introduced and it is shown that every ordinary star-like self-contained graph has infinitely many strong twins or none.
\end{abstract}

\section{Introduction}
Self-contained graphs are infinite graphs which have isomorphic copies of themselves as proper \textit{induced} subgraphs. These graphs were studied in \cite{Shekarriz} and, in this paper, we continue studying them by finding a special kind of self-contained graphs for which a renowned conjecture of Bonato and Tardif \cite{Bonato1} comes true. 

Self-contained graphs have fascinated mathematicians since 2003 by the so-called \emph{``Graph alternative conjecture''}, which has its origin in \cite{Bonato1} where Bonato and Tardif studied \emph{twins} of infinite graphs under the phrase \emph{``mutually embeddable graphs''}; two non-isomorphic graphs $G$ and $H$ are called \textit{``(strong) twins''} if $G$ is isomorphic to a proper (induced) subgraph of $H$ and $H$ is also isomorphic to a proper (induced) subgraph of $G$. They asked a question that if $G$ and $H$ are twins, then do $G$ and $H$ belong to an infinite family of twins? Three years later, they extended their study of twins in \cite{Bonato2} where they noted that if an infinite graph has a strong twin, then it is isomorphic to one of its proper induced subgraphs, i.e., in our phrase, every graph that has a strong twin is also self-contained. They also conjectured that in case of trees, the question has a positive answer. In other word,  they conjectured that every infinite tree has either infinitely many tree-twins or none. They called it \emph{``the tree alternative conjecture''} and proved it for rayless trees \cite{Bonato2}.

In 2009, Tyomkyn proved that the tree alternative conjecture is true for all rooted trees and also conjectured that, with the exception of the ray, every locally finite tree that is isomorphic to a proper subgraph of itself has infinitely many tree-twins \cite{Tyomkyn}. In 2011, another progress made by Bonato et. al. in \cite{rayless}, where they proved that (i) a rayless graph has either infinitely many twins or none, and (ii) a connected rayless graph has either infinitely many connected twins or none.

To read this paper, we need some few definitions, notations and results we have presented in \cite{Shekarriz}. Meanwhile, few definitions of infinite graph theory is also needed, all of which can be found in Section 8 of \cite{Diestel}. Moreover, to simplify, we use the notation $\emptyset$ for the null graph, the unique graph that has no vertices. Furthermore, we use the notations $\subset$ and $\sim_G$ respectively for subgraph and adjacency relations in a graph $G$, and, $G \setminus H$ always stands for the induced subgraph $G[V(G) \setminus V(H)]$ where $H$, itself, is an induced subgraph of $G$.

For a self-contained graph $G$, a non-empty proper subgraph $H$ is a \emph{removable subgraph} if $G \setminus H \cong G$. Then we write $H \in \mathrm{Rem}(G)$ and by $\mathrm{Iso}_{G} (H)$ we mean the set of all isomorphisms $f:G\longrightarrow G \setminus H$ \cite{Shekarriz}. We may also need the following two propositions:

\begin{proposition}
\label{union-rem}
Let $G$ be a self-contained graph, $P \in \mathrm{Rem}(G)$ and $Q$ be an induced subgraph of $G \setminus P$. Then $Q \in \mathrm{Rem}(G \setminus P)$ if and only if $P \cup Q \in \mathrm{Rem}(G)$ \textnormal{\cite{Shekarriz}}.
\end{proposition}
\begin{proposition}
\label{8} Let $G$ be a self-contained graph and $H \in \mathrm{Rem}(G)$. Then $G$ contains infinitely many vertex disjoint copies of $H$ \textnormal{\cite{Shekarriz}}. 
\end{proposition}

\section{The Result}
In this section, we find a category of self-contained graphs for which the graph alternative conjecture of Bonato and Tardif \cite{Bonato1} is true. In order to do this, we need the following statement whose proof is straightforward.

\begin{proposition}
\label{twin-constrain}
A graph $G$ has a strong twin if and only if $G$ is a self-contained graph which has a non-empty induced subgraph $P$ such that $P \notin \mathrm{Rem}(G)$ but there is $H \in \mathrm{Rem}(G)$ such that $P \subset H$.
\end{proposition}

We say $G$ has a \emph{strong twin trough $H$} if $H \in \mathrm{Rem}(G)$ and there is non-empty $P \subset H$ such that $P \notin \mathrm{Rem}(G)$.

\begin{lemma}
\label{star-removable}
Let $\{ H_{0}, H_{1}, H_{2}, \ldots \}$ be a family of mutually vertex-disjoint induced subgraphs of a graph $G$ and for each $i= 1, 2, \ldots$, there exists $\alpha_{i} \in \mathrm{Aut}(G)$ such that $\alpha_{i}(H_{0})=H_{i}$, $\alpha_{i}(H_{i})=H_{0}$ and $\alpha_{i}(v)=v$ for every other vertices of $G$. Then $G$ is a self-contained graph and $H_{0} \in \mathrm{Rem}(G)$.
\end{lemma}
\begin{proof}
We first note that for each $i,j=1,2, \ldots$, we have $\alpha_{i} \circ \alpha_{j} (H_{j}) = \alpha_{i}(H_{0})=H_{i}$. So, the following function is an automorphism of $G$:
$$\beta_{i,j}(v)= \left\{ \begin{array}{l l} 
\alpha_{i} \circ \alpha_{j} (v) & v \in H_{j}\\
\alpha_{j} \circ \alpha_{i} (v) & v \in H_{i} \\
v & v \notin H_{i} \cup H_{j}.
\end{array}\right.$$
Now put $\alpha_{0}=\mathrm{id}_{G}$ and define $f:G\longrightarrow G \setminus H_{0}$ with 
$$f(v)= \left\{ \begin{array}{l l}
\alpha_{i+1} \circ \alpha_{i} (v) & v \in H_{i}, i=0,1,2, \ldots\\
v & v \notin \cup_{i=0}^{\infty} H_{i}.
\end{array}\right.$$
We show that $f$ is an isomorphism between $G$ and $G \setminus H_{0}$ to deduce that $G$ is a self-contained graph and $H_{0} \in \mathrm{Rem}(G)$.

It is clear that $f$ is well-defined and one-to-one. To show that $f$ is onto, let $x$ be a vertex of $G \setminus H_{0}$. Then either $x \notin \cup_{i=1}^{\infty} H_{i}$ which means that $x=f(x)$ or there is a unique $i=1,2, \ldots$ that $x \in H_{i}$, for which we have $x=f\big( \beta_{i-1,i}(x) \big)$.

It remains to show that $f$ is adjacency preserving. Let $x \sim y$. Then there is three possibilities:
\begin{itemize}
\item[i.] $x,y \notin \cup_{i=0}^{\infty} H_{i}$. Then $f(x)=x \sim y = f(y)$.
\item[ii.] $x \in \cup_{i=0}^{\infty} H_{i}$ but $y \notin \cup_{i=0}^{\infty} H_{i}$, or vice versa. Then there is a unique $i=0,1,2,\ldots$ such that $x \in H_{i}$, and hence $f(x)=\alpha_{i+1} \circ \alpha_{i} (x)$. Since both $\alpha_i$ and $\alpha_{i+1}$ are automorphisms of $G$, we must have $y$ is adjacent to $f(x)$.
\item[iii.] $x,y \in \cup_{i=0}^{\infty} H_{i}$. So there are unique $i,j=0,1,2,\ldots$ that $x \in H_{i}$ and $y \in H_{j}$. If $i=j$ then $\alpha_{i+1} \circ \alpha_{i} (x)$ is adjacent to $\alpha_{i+1} \circ \alpha_{i} (y)$ i.e., $f(x) \sim f(y)$. If $i \neq j$ then $y$ is adjacent to $\alpha_{i}(x)$ which is adjacent to $\beta_{j+1,j}(y)=f(y)$ which must also be adjacent to $\alpha_{i+1} \big( \alpha_{i}(x) \big)=f(x)$.
\end{itemize}
Showing that $f$ preserves non-adjacencies is similar and completes the proof.
\end{proof}

Let $G$ be a self-contained graph and $H \in \mathrm{Rem}(G)$. We say $H$ is a \emph{well-mannered removable subgraph of $G$} if for each isomorphism $f \in \mathrm{Iso}_{G}(H)$ there exists an automorphism $\alpha \in \mathrm{Aut}(G)$ such that $f(H)=\alpha(H)$, $\alpha^{2}(H)=H$ and $\alpha(v)=v$ for all $v \notin H \cup f(H)$. In this case, we also say $a$ is an \emph{alternating automorphism for $H$ and $f(H)$} or more conveniently, $a$ is an \emph{alternating automorphism for $f$}. Moreover, we may sometimes say that $f(H)$ is \emph{an alternating copy of $H$ in $G$}. Furthermore, we say $G$ is \emph{star-like} if all of its removable subgraphs are well-mannered.

Let us consider some useful properties of well-mannered removable subgraphs. When $G$ is a self-contained graph, $H \in \mathrm{Rem}(G)$, $f \in \mathrm{Iso}_{G}(H)$ and $\alpha \in \mathrm{Aut}(G \setminus H)$, we are able to add a copy of $H$ to $\alpha(G \setminus H)$ and obtain an isomorphic copy of $G$. In this case we say that $H$ is \emph{sewed} to $\alpha(G \setminus H)$ and $f^{-1} \circ \alpha^{-1}$ is an isomorphism between $\alpha(G \setminus H)$ and $G$. In particular, when $H$ is well-mannered, by iteratively removing and sewing copies of $H$, it can be shown that there is an infinite family $\mathcal{A}$ of mutually vertex-disjoint copies of $H$ that the formation of Lemma \ref{star-removable} holds for $G$ and $H$ and for each countable subfamily $\mathcal{R}$ of $\mathcal{A}$ containing $H$. Therefore, there is an \emph{standard isomorphism} $g \in \mathrm{Iso}_{G}(H)$ like what is introduced in the proof of Lemma \ref{star-removable} that only moves $\mathcal{R}_g=\{ H=H_{0}, H_{1}, H_{2}, \ldots \}$. Moreover, the following proposition states one of the most important properties of well-mannered removable subgraphs:

\begin{proposition}
Let $G$ be a self-contained graph and $H$ be a well-mannered removable subgraph of $G$. Then for each isomorphism $f \in \mathrm{Iso}_{G}(H)$ we have $f(H) \in \mathrm{Rem}(G)$ and there exists isomorphism $g \in \mathrm{Iso}_{G}(f(H))$ such that $g(f(H))=H$.
\end{proposition}
\begin{proof}
Let $\alpha \in \mathrm{Aut}_{G}(H)$ such that $f(H)=\alpha(H)$, $\alpha^{2}(H)=H$ and $\alpha(v)=v$ for all $v \notin H \cup f(H)$. Then, $g=\alpha \circ f \circ \alpha$ is an isomorphism between $G$ and $G \setminus f(H)$ such that $g(f(H))=H$.
\end{proof}

Let $G$ be a self-contained graph and $H \in \mathrm{Rem}(G)$. A vertex $v$ of $G$ is called a \emph{twisted vertex} for $H$ if there exists $P \in \mathrm{Rem}(G)$ such that $v \in V(P)$ but $v \notin V(Q)$ for all $Q \in \mathrm{Rem}(G\setminus H)$. The subgraph induced by all twisted vertices for $H$ is called the \emph{torsion} of $H$ and is denoted by $\mathrm{Tor}_{G}(H)$.  Meanwhile, when $\mathrm{Tor}_{G} (H) = \emptyset$ we say $H$ is a \emph{torsion-free removable subgraph} of $G$. For some examples and implications of torsion subgraphs, see~\cite{Shekarriz}. Here, we show that every well-mannered removable subgraph is torsion-free:

\begin{theorem}
Let $G$ be a self-contained graph and $H$ be a well-mannered removable subgraph of $G$. Then $\mathrm{Tor}_{G}(H)=\emptyset$.
\end{theorem}
\begin{proof}
Let $f \in \mathrm{Iso}_{G}(H)$ be a standard isomorphism, then $$f\big( \mathrm{Tor}_{G}(H) \big) =\mathrm{Tor}_{G \setminus H}\big( f(H) \big)= \mathrm{Tor}_{G}(H)$$ because $f$ fixes vertices outside $\mathcal{R}_f$. By the way, if $v \in \mathrm{Tor}_{G}(H)$, it is an asset vertex to $G \setminus H$ and cannot belong to a removable subgraph in $G \setminus H$, i.e., $v \notin \mathrm{Tor}_{G \setminus H}\big( f(H) \big)$. So, there is no such a $v$ and we must have $\mathrm{Tor}_{G}(H)=\emptyset$.
\end{proof}

Let $G$ be a star-like self-contained graph which has a strong twin trough $H$, $f \in \mathrm{Iso}_{G}(H)$, $\alpha_{f} \in \mathrm{Aut}(G)$ be the alternating automorphism of $f$, and, $\beta_{i,j} \in \mathrm{Aut}(G)$ be the automorphism that alternates $f^{i}(H)$ and $f^{j}(H)$ and fixes other vertices. So, by Proposition \ref{twin-constrain}, there is a non-empty $P \subset H$ such that $P \notin \mathrm{Rem}(G)$ and $G_{1}=G \setminus P$ is a twin of $G=G_{0}$. By Proposition \ref{union-rem}, it is also clear that $Q= H \setminus P$ is not a removable subgraph of $G_{1}$. 

By the way, for $i=2, 3, \ldots$, put $G_{i} = G \setminus \bigcup_{j=1}^{i} f^{j-1}(P)$. The restriction of $\alpha_f$ to $G_{i}$, namely $\overline{\alpha_f}$, is an automorphism of $G_i$ such that $\overline{\alpha_f}(Q)=f(Q)$, $\overline{\alpha_f}^{2}(Q)=Q$ and $\overline{\alpha_f}(v)=v$ for all $v \in G_{i} \setminus \big( Q \cup f(Q)\big)$.

In the following Lemma, we show that $G_{2}, G_{3}, \ldots$ are all strong twins for $G$.

\begin{lemma}
\label{star-twins}
Let $G, G_{1}, G_{2}, \ldots$ be the above described graphs. Then $G_{2}, G_{3}, \ldots$ are all strong twins for $G$.
\end{lemma}

\begin{proof}
Since $G, G_{1}, G_{2}, \ldots$ are mutually embeddable, we only show that they are all non-isomorphic to $G$. 

Suppose on contrary that there is an $i=2, 3 , \ldots$ such that $G \simeq G_{i}$. Therefore, $W=\bigcup_{j=1}^{i}f^{j-1}(P)$ is a well-mannered removable subgraph of $G$. On the other hand, $M=\bigcup_{j=1}^{i}f^{j-1}(Q)$ is also a well-mannered removable subgraph of $G_{i}$.  

Put $X=Q \cup f^{i}(P)$. Since the restriction of $f$ to $G_i$, namely $f^*$, is an isomorphism from $G_i$ to $G_{i} \setminus X$, we must have $X$ is a removable subgraph of $G_i$. We show that $X$ is not a well-mannered removable subgraph of $G_i$, contradicting the assumption $G \simeq G_{i}$.

If $X$ is a well-mannered removable subgraph of $G_i$, there must be an alternating automorphism $\gamma \in \mathrm{Aut}(G_{i})$ such that $\gamma(X)=f^{*}(X)$, $\gamma^{2}(X)=X$ and $\gamma$ fixes all other vertices of $G_i$. On the other hand, as noted above, the restriction of $\alpha_f$ to $G_i$, namely $\overline{\alpha_f}$, is an automorphism of $G_i$ which alternates $Q$ and $f(Q)$. Therefore, there is another automorphism $\xi=\overline{\alpha_f} \circ \gamma$ which alternates $f^{i}(P)$ and $f^{i+1}(P)$ and fixes all other vertices. The automorphism $\xi$ can be lifted to an automorphism $\overline{\xi}$ of $G$ which also fixes vertices of $W$. Now, $\beta_{0,i}\circ\beta_{j,i+1}\circ\overline{\xi}\circ\beta_{0,i}\circ\beta_{j,i+1}$ is an automorphism of $G$ that alternates $P$ and $f^{j}(P)$ and fixes other vertices, for $j=1,2, \ldots$. So, $\mathcal{A}= \{P, f(P), f^{2}(P), \ldots \}$ is an infinite family of mutually vertex-disjoint alternating copies of $P$ in $G$, and thus by Lemma \ref{star-removable}, $P$ is a removable subgraph of $G$, a contradiction.
\end{proof}

Since all $G_i$s are mutually embeddable, if we were able to prove that $G_{1}, G_{2}, \ldots$ are also mutually non-isomorphic, we had been arrived to a proof for graph alternative conjecture for all star-like self-contained graphs. Although it is quite tempting to try this in the general case, the following example shows that it is even possible that all $G_{2}, G_{3}, \ldots$ be isomorphic to $G_1$.

\begin{example}
\label{extended-star}
Let $G$ be a graph defined as follows: $V(G)=A_{1}\cup A_{2} \cup \{ o \}$ where $A_{i}=\{ a_{i,1}, a_{i,2}, \ldots \}$ for $i=1,2$. And, for edges of $G$ we have $a_{1,j}$ is adjacent to $a_{2,j}$ and $o$ for each $j \in \mathbb{N}$. Then, $G$ is a star-like self-contained graph. Let $\mathfrak{p}_n$ be the $n^{\mathrm{th}}$ prime number and $P=\{ a_{2,2^j} \vert j \in \mathbb{N} \}$. We then have $P \notin \mathrm{Rem}(G)$ but $P \subset H=\{ a_{i,2^j} \vert j \in \mathbb{N},i=1,2 \} \in \mathrm{Rem}(G)$. Now put $G_{1}= G \setminus P$ which can easily be recognized as a strong twin of $G$. Let $f:G \longrightarrow G \setminus H$  be the isomorphism that moves $a_{i,{\mathfrak{p}_{n}}^j}$ to $a_{i,{\mathfrak{p}_{n+1}}^j}$ for $j \in \mathbb{N}$ and $i=1,2$ and fixes all other vertices. Now if we construct $G_{2}, G_{3}, \ldots$ like what is said before Lemma \ref{star-twins}, we have $G_{k} \simeq G_{1}$ for all $k=2,3, \ldots$.
\end{example}

The obstacle we faced in Example \ref{extended-star} is that $Q=H \setminus P$ is a self-contained graph which has a removable subgraph isomorphic to itself! If we could guarantee that this case does not happen for a specific star-like self-contained graph $G$, we can proceed to prove the conjecture for $G$. In particular, if for each removable subgraph $H$ which contains a non-removable subset $P$, there exists non-empty $P^{'} \subset H$ such that $Q^{'}=H \setminus P^{'}$ is a finite graph, then the cases similar to Example \ref{extended-star} can be replaced by some well-behaved cases, and, we say that $G$ is an \emph{ordinary star-like self-contained graph}. Moreover, When $G$ has a strong twin, namely $G_1$ such that $G_1$ contain a finite graph $Q$ for which we have $G \simeq G_{1} \setminus Q$, we say $G_1$ is an \emph{ordinary strong twin} for $G$.

\begin{theorem}
\label{orderly-twin}
Let $G$ be an ordinary star-like self-contained graph. Then $G$ has infinitely many strong twins or none.
\end{theorem}
\begin{proof}
If $G$ does not have a strong twin, there is nothing to prove. So, suppose $G$ has a strong twin $G_{1}=G \setminus P$ trough a removable subgraph $H \in \mathrm{Rem}(G)$ and let $Q$, $f$, $G_{2}, G_{3}, \ldots$ be defined like those right before Lemma \ref{star-twins}, and, as above, we can assume that $Q$ is a finite graph. Since by Lemma \ref{star-twins}, $G_{1}, G_{2}, \ldots$ are all strong twins for $G$ and each pair of them contain mutual embedding, we only need to show that they are mutually non-isomorphic.

Suppose on the contrary that there are natural numbers $i$ and $j$ such that $i<j$ and $G_{i} \simeq G_{j}$. Then $M=\bigcup_{k=1}^{j-i}f^{k-1}(Q)$ is a finite removable graph to $G_j$. Let $g:G_{j} \longrightarrow G_{j} \setminus M$ be an isomorphism. Since in $G_i$ and $G_j$ there are $i$ and $j$ vertex disjoint alternating copies of $Q$, respectively, and because $i<j$ and $Q$ is a finite graph, it can be deduced that $g(Q)$ is outside alternating copies of $Q$ in $G_j$. Therefore, $g(Q)$, which is an induced subgraph of $G$, has $j$ alternating copies in $G$. Hence, if we put $Y=f^{j+1}(P)$, then we must have $X=g(Q) \cup g\big( Y \big)$ is a removable subgraph of $G$, there is isomorphism $\ell:G \longrightarrow G \setminus X$ such that $\ell(X)= g\big( f(Q) \big) \cup g\big( f(Y) \big)$. Now, with an argument similar to the proof of Lemma \ref{star-twins}, we must have $X$ is not well-mannered, contrary to the fact that $G$ is a star-like self-contained graph.
\end{proof}

Now it is time to prove a connected version of the graph alternative conjecture for ordinary star-like self-contained graphs.

\begin{theorem}
Let $G$ be a connected ordinary star-like self-contained graph which has a connected ordinary strong twin. Then $G$ has infinitely many connected strong twins.
\end{theorem}

\begin{proof}
Since $G$ is ordinary star-like and has a strong twin, by Theorem \ref{orderly-twin}, it has infinitely many twins like those constructed in the proof. So, with the terminology of the proof of Theorem \ref{orderly-twin} and Lemma \ref{star-twins} for $G_{i}$s, $H, P, Q, f$ and $a_{f}$, we inductively show that all $G_{i}$s are connected provided that $G$ and $G_{1}$ are both connected. To do this, we only replace $f:G \longrightarrow G \setminus H$ with the standard isomorphism $f^{*}:G \longrightarrow G \setminus H$ that only moves $H$ to $f^{2}(H)$, $f^{j}(H)$ to $f^{j+1}(H)$ for $j=2,3, \ldots$, and fixes $H^{*}=f(H)$ and all other vertices.

Suppose that $G, G_{1}, \ldots, G_{i-1}$ are all connected for $i=2,3, \ldots$. Then $$G_{i} = G_{i-1} \setminus {f^{*}}^{i-1}(P) = G \setminus \bigcup_{j=0}^{i-1} {f^{*}}^{j}(P) = G \setminus W.$$
But $G_{i} \setminus \bigcup_{j=0}^{i-1} {f^{*}}^{j}(Q) = G_{i} \setminus M$ is an isomorphic copy of $G$ in $G_{i}$ which contains $H^{*}$ and is connected. So, in $G_{i}$, every vertices $v \in V(H^{*})$ has a path to all other vertices of  $G_{i} \setminus M$, and, since $G_{1}, \ldots, G_{i-1}$ are all connected and ${f^{*}}^{i-1}(H)$ is a removable subgraph to all these self-contained graphs, every vertices of $\bigcup_{j=0}^{i-2} {f^{*}}^{j}(Q)$ has a path to $v$ which does not meet $\bigcup_{j=0}^{i-1}{f^{*}}^{j}(H)$. Therefore, it is only needed to show that every vertices of ${f^{*}}^{i-1}(Q)$ has a path to $v$ that does not meat ${f^{*}}^{i-1}(P)$.

Let $\beta_{0(i-1)} \in \mathrm{Aut}(G)$ be the automorphism that $$\beta_{0(i-1)}(H)={f^{*}}^{i-1}(H), \beta_{0(i-1)}\big( {f^{*}}^{i-1}(H)\big) =H$$ and fixes all other vertices of $G$. Then the restriction of $\beta_{0(i-1)}$ to $G_{i}$ is an automorphism of $G_{i}$ that alternates $Q$ with ${f^{*}}^{i-1}(Q)$. Consequently, the desired paths are images of already assumed paths from vertices of $Q$ to $v$.
\end{proof}

The reader should note that although they have some overlaps, there are infinitely many ordinary star-like self-contained graphs which are neither rayless nor rooted trees. For instance, let $G$ be a graph consisting of countably many disjoint copies of $K_{\aleph_0}$, i. e., $\{ K_{\aleph_0}^{i} : i \in \mathbb{N} \}$ along with a single vertex $o$, and, $g_{i}:\mathbb{N}\longrightarrow K_{\aleph_0}^{i}$ be fixed enumerations. Let also $o$ be adjacent to each vertex of $K_{\aleph_0}^{1}$ and every vertex $v$ of $K_{\aleph_0}^{i}$ be adjacent to $u$ of $K_{\aleph_0}^{i+1}$ if $g_{i}^{-1}(v)= g_{i+1}^{-1}(u)$. The graph $G$ is then an ordinary star-like self-contained graph for which $H_{1}=\{ g_{j}(1) : j \in \mathbb{N} \}$ is a removable subgraph and $f:G \longrightarrow G\setminus H$ defined by $$f(v)=\left\{\begin{matrix}
o & v=o\\
g_{j}\big( g_{j}^{-1}(v)+1 \big) & v \in K_{\aleph_0}^{j}
\end{matrix} \right. $$ is an isomorphism. Now, let $Q=\{ g_{1}(1) \}$ and $P= H \setminus Q$. Then $G_{1} =G \setminus P$ is a strong twin for $G$. Consequently, $G_{2}, G_{3}, \ldots$, which were constructed prior to Lemma \ref{star-twins}, are different classes of twins for $G$. However, $G$ is neither rayless, nor a rooted tree.


\end{document}